\documentclass[12pt]{article}
\usepackage{amsmath,amssymb,amsthm}
\usepackage[a4paper, top=3cm, bottom=3cm, left=2.5cm, right=2.5cm]{geometry}

 \def\draw #1 by #2 (#3){
  \vbox to #2{
    \hrule width #1 height 0pt depth 0pt
    \vfill
    \special{picture #3} 
    }
  }

 \def\scaleddraw #1 by #2 (#3 scaled #4){{
  \dimen0=#1 \dimen1=#2
  \divide\dimen0 by 1000 \multiply\dimen0 by #4
  \divide\dimen1 by 1000 \multiply\dimen1 by #4
  \draw \dimen0 by \dimen1 (#3 scaled #4)}
  }
\usepackage{amsfonts}
\usepackage{amssymb}
\newtheorem{thm}{Theorem}[section]

\newtheorem{lem}[thm]{Lemma}
\newtheorem{prop}[thm]{Proposition}
\theoremstyle{definition}
\newtheorem{defn}[thm]{Definition}
\newtheorem{ex}[thm]{Example}
\newtheorem{rem}[thm]{Remark}
\numberwithin{equation}{section}

\title{\bf Set-valued mapping and Rough Probability }
\author{{\small Shaban Sedghi}\small $^1$, {\small Nabi Shobe}\small $^2$,  {\small Dae-Won Lee\small$^3$\thanks{e-mail : haverd2001@gmail.com}~ and {\small Siamak Firouzian}\small $^4$}\\
\small $^1$Department of Mathematics, Qaemshahr Branch, \\
\small Islamic Azad University, Qaemshahr, Iran\\
\small sedghi\_gh@yahoo.com\\
\small $^2$Department of Mathematics, Babol Branch, \\
\small Islamic Azad University, Babol,  Iran\\
\small nabi\_shobe@yahoo.com\\
\small $^3$ Department of Mathematics, \\
 \small Yonsei University, Seoul 120-749, Republic of Korea.\\
\small $^4$Department of Mathematics, Payame noor University, Iran\\
\small siamfirouzian@pnu.ac.ir}

\date{}
\begin{document}
\maketitle
\begin{abstract}
 In 1982, the theory of rough sets proposed by Pawlak and in 2013, Luay concerned a rough probability by using the notion of Topology. In this paper, we study the rough probability in the stochastic approximation spaces by using set-valued mapping and obtain results on rough expectation, and rough variance. \\ \\
{\bf  AMS Mathematics Subject Classification:} 54E40, 54E35, 54H25. \\
{\bf Keywords:} Rough set; Lower inverse approximation, Upper inverse approximation, Set-valued mapping, Stochastic approximation space, rough expectation, rough variance.

\end{abstract}

\section{Introduction and Preliminaries}

 The theory of rough sets was first introduced by Pawlak \cite{17}. Rough set theory, a new mathematical approach to deal with inexact, uncertain or vague knowledge, has
recently received wide attention on the research areas in both of the real-life applications and the theory itself. Also, after the proposal by Pawlak, there have been many researches on the connection between rough sets and algebraic systems \cite{1,2,3,4,5,6,8,9,10,12,13,14}. In \cite{m.j} Jamal study stochastic approximation spaces from topological view that generalize the stochastic approximation space in the case of general relation. The couple  $S =(X, P)$ is called the stochastic approximation space, where X is a non-empty set and $P$ is a probability measure.

\vspace{3mm}
 Lower and upper inverse is defined as follows:

 \begin{defn}
 Let X  be a non-empty set and $A \subseteq X$. Let $T :X\longrightarrow P^* (X )$  be a set-valued mapping where $P^* (X)$ denotes the set of all non-empty subsets of X. The lower
 inverse and upper inverse of A under T are defined as \[T^+ (A)=\{x\in X |T (x )\subseteq A\}~~ ;~~ T^{-1} (A )=\{ x \in X |T (x) \cap A \not =\emptyset\},\] respectively. Also,  $(T^+(A) ,T^{-1} (A))$ is called T-rough set of X.
 \end{defn}

 \vspace{3mm}

 Also, using lower and upper inverse, the lower and upper probability is defined as follows:

 \begin{defn}\label{defn1}
 Let $T :X\longrightarrow P^* (X )$  be a set-valued mapping and $A$ be an event in the stochastic approximation space $S = (X,P)$. Then the lower and upper probability of A is given by: \[\underline{P}(A) = P(T^+(A))~~ ;~~ \overline{P}(A) = P(T^{-1}(A)),\] respectively.  Clearly, $0\leq \underline{P}(A)\leq 1$ and $0\leq \overline{P}(A)\leq 1$.
 \end{defn}

 \vspace{3mm}

 \begin{defn}
 Let X be a non-empty set. Let $T :X\longrightarrow P^* (X )$ be a set-valued mapping. Then we say $T$ has

 i) reflective property, if for every $x\in X$ we have $x\in T(x)$,

 ii) transitive property, if for every $y\in T(x)$ and $z\in T(y)$ we have $z\in T(x)$.
 \end{defn}

 \vspace{3mm}

  \begin{rem}
 Let  $T$ has reflective and transitive properties, then in topological space $(X,\tau)$ we have $T^{-1}(A)=\overline{A}$ and $T^+(A)=A^{o}$, where $A^{o}$ denotes interior of $A$ and $\overline{A}$ denotes the closure of $A$. These implies that  Definition \ref{defn1} of our paper is same the Definition 2.2 of paper \cite{a.a.a}. Hence this paper is generalized version of paper \cite{a.a.a}.
 \end{rem}

 \vspace{3mm}

 \begin{defn}
 Let  $A$ be a subset of topological space   $(X,\tau)$, then we called $A$ is a exact set if $T^+(A)=T^{-1}(A)=A$.
 \end{defn}

 \vspace{3mm}

 \section{Main Result}

 \begin{prop} Let $T:X\longrightarrow P^* (X )$  be a set-valued mapping and  $A,B$ be two events in the  stochastic approximation space $S=(X,P)$.
 Then the following holds:
 \begin{itemize}
 \item[{\rm (1)}] $\underline{P}(\emptyset)=0=\overline{P}(\emptyset)$;
 \item[{\rm (2)}] $\underline{P}(X)=1=\overline{P}(X)$;
 \item[{\rm (3)}] $\overline{P}(A\cup B)\leq \overline{P}(A)+ \overline{P}(B)-\overline{P}(A\cap B)$;
 \item[{\rm (4)}] $\underline{P}(A\cup B)\geq \underline{P}(A)+\underline{P}( B)-\underline{P}(A\cap B)$;
 \item[{\rm (5)}] $\underline{P}( A^c)=1-\overline{P}(A)$;
 \item[{\rm (6)}] $\underline{P}( A- B)\leq \underline{P}(A)-\underline{P}(A\cap B)$;
 \item[{\rm (7)}] $\underline{P}( A)\leq \overline{P}(A)$;
 \item[{\rm (8)}] If $A\subseteq B$, then $\underline{P}(A)\leq \underline{P}(B)$ and $\overline{P}(A)\leq \overline{P}(B)$.
 \end{itemize}
 \end{prop}
 \begin{proof} It is straightforward.
 \end{proof}

 \vspace{3mm}

 \begin{defn} Let $T:X\longrightarrow P^* (X )$  be a set-valued mapping and $A$ be an event in the  stochastic approximation space $S = (X, P)$. The rough probability of $A$, denoted by $P^*(A)$, is given by: \[P^*(A)=(\underline{P}(A),\overline{P}(A)).\]
 \end{defn}

 \vspace{3mm}

 \begin{lem}Let $T:X\longrightarrow P^* (X )$  be a set-valued mapping and  $A$ be a event in the  stochastic approximation space $S=(X,P)$.
 \begin{itemize}
 \item[{\rm (1)}] If $T$ has reflective property, then $\underline{P}(A)\leq P(A)\leq \overline{P}(A)$;
 \item[{\rm (2)}]If $T$ has reflective and transitive properties, then $\underline{P}(T^+(A))= \underline{P}(A)$ and $ \overline{P}(T^{-1}(A))=\overline{P}(A)$;
 \item[{\rm (3)}]If  $A$ is an exact subset of $X$, then $\underline{P}(A)= P(A)= \overline{P}(A)$.
 \end{itemize}
 \end{lem}
 \begin{proof}
 It is straightforward.
 \end{proof}

 \vspace{3mm}

 \begin{ex}\label{ex1} Let $X=\{1,2,3,4,5,6\}$ and let  $T : X \longrightarrow P^*(X)$ where for every $n \in X$, $T(1) = \{1\}, T(2) = \{1,2\}, T(3) = \{3\}, T(4) = \{4\}, T(5) =T(6) = \{1,5,6\} $.

 \begin{itemize}
 \item[{\rm (1)}] Let $A=\{1,3,5\}$ then $T^+(A)=\{1,3\}$, $\underline{P}(A)=\frac{2}{6}$ and $T^{-1}(A)=\{1,2,3,5,6\}$, $P(A)=\frac{3}{6}$ and $\overline{P}(A)=\frac{5}{6}$.
 \end{itemize}
 \end{ex}

 Tabel 2.1: Lower and upper probabilities of a random variable U\\
 \begin{center}
 \begin{tabular}{|c|c|c|c|c|c|c|}
 \hline
  u &~ 1 ~& 2 ~& 3 ~& 4 ~& 5 ~& 6 \\\hline
  $\underline{P}(U=u)$ & $\frac{1}{6}$ & 0 & $\frac{1}{6}$ & $\frac{1}{6}$ & 0 & 0 \\\hline
  $\overline{P}(U=u)$ & $\frac{4}{6}$ & $\frac{1}{6}$ & $\frac{1}{6}$ & $\frac{1}{6}$ & $\frac{2}{6}$ & $\frac{2}{6}$ \\ \hline
 \end{tabular}
 \end{center}

 \vspace{3mm}

 \begin{defn} Let $T:X\longrightarrow P^* (X )$  be a set-valued mapping and  $A,B$ be two events in the  stochastic approximation space $S=(X,P)$. We define $ \underline{P}(A|B)= \frac{\underline{P}(A\cap B)}{\underline{P}(B)}$ for every $\underline{P}(B)\neq 0$ and $\overline{P}(A|B)=\frac{\overline{P}(A\cap B)}{\overline{P}(B)}$ for every $\overline{P}(B)\neq 0$.
 \end{defn}

\vspace{3mm}

 \begin{lem} Let $T:X\longrightarrow P^* (X )$  be a set-valued mapping and  $A,B,C$ be three events in the  stochastic approximation space $S=(X,P)$.
 then the following holds:
 \begin{itemize}
 \item[{\rm (1)}] $\underline{P}(A|A)=\overline{P}(A|A)=1$;
 \item[{\rm (2)}] $\underline{P}(\emptyset|A)=\overline{P}(\emptyset|A)=0$;
 \item[{\rm (3)}] $\underline{P}(A|X)=\underline{P}(A)$ and $\overline{P}(A|X)=\overline{P}(A)$;
 \item[{\rm (4)}] $\underline{P}(A^c| B)\leq 1-\underline{P}(A|B)$;
 \item[{\rm (5)}] $\underline{P}(A\cup B|C)\geq \underline{P}(A|C)+\underline{P}( B|C)-\underline{P}(A\cap B|C)$;
 \item[{\rm (6)}] $\overline{P}( A^c|B)\geq 1-\overline{P}(A|B)$;
 \item[{\rm (7)}] $\overline{P}(A\cup B|C)\leq \overline{P}(A|C)+\overline{P}( B|C)-\overline{P}(A\cap B|C)$;
 \item[{\rm (8)}] $\underline{P}(A)\geq\sum_{i=1}^{n}\underline{P}(A|B_i)\underline{P}(B_i)$, where $\bigcup_{i=1}^{n}B_i=X$;
 \item[{\rm (9)}] $\overline{P}(A)\leq\sum_{i=1}^{n}\overline{P}(A|B_i)\overline{P}(B_i)$, where $\bigcup_{i=1}^{n}B_i=X$;
 \item[{\rm (10)}] If $T$ has transitive property and if $B$ is an exact subset of $X$ then \[\underline{P}(A|B)\leq P(A|B)\leq \overline{P}(A|B).\]
 \end{itemize}
 \end{lem}

 \vspace{3mm}

 \begin{ex}Consider the same experiment as in Example \ref{ex1}. Let $X=\{1,2,3,4,5,6\}$, $B=\{1,3,5\}$ and $A=\{4,5,6\}$ then \[\underline{P}(A|B)=\frac{\underline{P}(\{5\})}{\underline{P}(\{1,3,5\})}=0,\] and \[\overline{P}(A|B)=\frac{\overline{P}(\{5\})}{\overline{P}(\{1,3,5\})}=\frac{\frac{2}{6}}{\frac{5}{6}}=\frac{2}{5}.\]
 \end{ex}

 \vspace{3mm}

 We define the lower and upper distribution functions of a random variable $U$.

 \begin{defn}Let $T:X\longrightarrow P^* (X )$  be a set-valued mapping and $U$ be a random variable in the  stochastic approximation space $S = (X,P)$. The lower and upper distribution of $U$ is given by: \[\underline{F}(u)=\underline{P}(U\leq u)~~~;~~~ \overline{F}(u)=\overline{P}(U\leq u),\] respectively.
 \end{defn}

 \vspace{3mm}

\begin{defn} Let $T
:X\longrightarrow P^* (X )$  be a set-valued mapping and $U$ be a
random variable in the  stochastic approximation space $S=(X,P)$.
The rough distribution function of $U$, denoted by $F^*(u)$, is
given by:
\[F^*(u)=(\underline{F}(u),\overline{F}(u)).\]\end{defn}

\vspace{3mm}

\begin{ex}
Consider the same experiment as in Example \ref{ex1}. The Lower
and upper distribution function of $U$ are
\begin{equation*}
\underline{F}(u)=\left\{
\begin{array}{cc}
0, & -\infty<u<1, \\
 \frac{1}{6} , & 1\leq u<3 , \\
\frac{2}{4}, & 3\leq u<4 , \\
\frac{3}{6}, & 4\leq u<\infty
\end{array}%
\right.
\end{equation*}And
\begin{equation*}
\overline{F}(u)=\left\{
\begin{array}{cc}
0, & -\infty<u<1, \\
\frac{4}{6}, & 1\leq u<2 , \\
\frac{5}{6}, & 2\leq u<3, \\
1, & 3\leq u<4, \\
\frac{8}{6}, & 4\leq u<5, \\
\frac{9}{6}, & 5\leq u<6, \\
\frac{11}{6}, & 6\leq u<\infty.
\end{array}%
\right.
\end{equation*}Therefore $F^*(2)=(\frac{1}{6}, \frac{5}{6}).$
\end{ex}

\vspace{3mm}

We define the lower and upper expectations of a random variable
$U$ in the  stochastic approximation space $S=(X,P)$.

\begin{defn}   Let $T
:X\longrightarrow P^* (X )$  be a set-valued mapping and $U$ be a
random variable in the
 stochastic approximation space $S=(X,P)$.
The lower and upper expectation of $U$ is given by:
\[\underline{E}(u)=\sum_{k=1}^{n}u_k \underline{P}(U=u_k)~~~;~~~\overline{E}(u)=\sum_{k=1}^{n}u_k \overline{P}(U=u_k),\]respectively.
\end{defn}

\vspace{3mm}

\begin{defn}  Let $T
:X\longrightarrow P^* (X )$  be a set-valued mapping and $U$ be a
random variable in the  stochastic approximation space $S=(X,P)$.
The rough expectation of $U$ is denoted by $E^*(U)$ and is given
by: \[E^*(U)=(\underline{E}(U), \overline{E}(U)).\]\end{defn}

\vspace{3mm}

\begin{ex}
 Consider the same experiment as in Example \ref{ex1}. Then the
lower and upper expectations of $U$ are
\[\underline{E}(U)=1\cdot \frac{1}{6}+3\cdot \frac{1}{6}+4\cdot \frac{1}{6}=\frac{4}{3},\]and \[\overline{E}(U)=
1\cdot \frac{4}{6}+2\cdot \frac{1}{6}+3\cdot \frac{1}{6}+4\cdot \frac{1}{6}+5\cdot \frac{2}{6}+6\cdot \frac{2}{6}=\frac{35}{6}.\]
Hence rough expectation of $U$ is\[E^*(U)=(\frac{4}{3},
\frac{35}{6}).\]
\end{ex}

\vspace{3mm}

\begin{thm}\label{them} Let $T
:X\longrightarrow P^* (X )$  be a set-valued mapping and $U$ be a
random variable in the  stochastic approximation space $S=
(X,P)$. For any constants $a$ and $b$, we have
\[\underline{E}(aU+b ) = a\underline{E}(U)+bc~where~0\leq c\leq
1.\]
\end{thm}
\begin{proof}
\begin{eqnarray*}\underline{E}(aU+b )
&=&\sum_{k=1}^{n}(au_{k}+b)\underline{P}(u_{k})=\sum_{k=1}^{n}(au_{k}\underline{P}(u_{k})+b\underline{P}(u_{k})\\&=&
a\sum_{k=1}^{n}u_{k}\underline{P}(u_{k})+b\sum_{k=1}^{n}\underline{P}(u_{k})\\&=&a\underline{E}(U)+bc~~where~~
c=\sum_{k=1}^{n}\underline{P}(u_{k})(i.e~~ 0\leq c\leq 1).
\end{eqnarray*}
\end{proof}

\vspace{3mm}

\begin{thm} Let $T
:X\longrightarrow P^* (X )$  be a set-valued mapping and $U$ be a
random variable in the  stochastic approximation space $S=
(X,P)$. For any constants $a$ and $b$, we have
\[\overline{E}(aU+b ) = a\overline{E}(U)+bd~where~1\leq d\leq
n, ~~~n\in\mathbb{N}.\] \end{thm}
\begin{proof}The proof is
similar to Theorem \ref{them}.
\end{proof}

\vspace{3mm}

We define the lower and upper variances of a random variable $U$
in the stochastic approximation space $S = (X,P )$.

\begin{defn} Let $T
:X\longrightarrow P^* (X )$  be a set-valued mapping and $U$ be a
random variable in the
 stochastic approximation space $S = (X,P )$.
The lower and upper variance of $U$ is given by:
\[\underline{V}(U)=\underline{E}(U-\underline{E}(U))^2 ~~~;~~~ \overline{V}(U)=\overline{E}(U-\overline{E}(U))^2,\] respectively.
\end{defn}

\vspace{3mm}

\begin{defn} Let $T
:X\longrightarrow P^* (X )$  be a set-valued mapping and $U$ be a
random variable in the
 stochastic approximation space $S = (X,P )$. The
rough variance of $U$ is denoted by $V^*(U)$ and is given by:
\[V^*(U)=(\underline{V}(U), \overline{V}(U)).\] \end{defn}

\vspace{3mm}

\begin{ex} Consider the same experiment as in Example \ref{ex1}.
Then the lower and upper variances of $U$ are
\[\underline{V}(U)=0.4,~~\overline{V}(U)=13.75.\]The rough variance of $U$ is $V^*(U)=(0.4, 13.75).$\end{ex}

\vspace{3mm}

\begin{thm}\label{them1} Let $T
:X\longrightarrow P^* (X )$  be a set-valued mapping and $U$ be a
random variable in the  stochastic approximation space $S=
(X,P)$. Then
\[\underline{V}(U)=\underline{E}(U)^2-(2-c)(\underline{E}(U))^2~~~
where~~~ c =
\sum_{k=1}^{n}\underline{P}(u_k).\]\end{thm}
\begin{proof}
 We have \begin{eqnarray*}\underline{E}(U-\underline{E}(U) )^2
&=&\underline{E}(U^2-2U\underline{E}(U)+(\underline{E}(U)
)^2)\\&=&
\underline{E}(U)^2-2\underline{E}(U)\underline{E}(U)+c(\underline{E}(U))^2
~~~where~~~c=\sum_{k=1}^{n}\underline{P}(u_k)\\&=&\underline{E}(U)^2-2(\underline{E}(U))^2+c(\underline{E}(U))^2\\&=&
\underline{E}(U)^2-(2-c)\underline{E}(U)^2.
\end{eqnarray*}\end{proof}

\vspace{3mm}

\begin{thm} Let $T
:X\longrightarrow P^* (X )$  be a set-valued mapping and $U$ be a
random variable in the  stochastic approximation space $S=
(X,P)$. Then
\[\overline{V}(U)=\overline{E}(U)^2-(2-d)(\overline{E}(U))^2~~~ where~~~
d = \sum_{k=1}^{n}\overline{P}(u_k).\]\end{thm}\begin{proof}
 The proof is similar to Theorem \ref{them1}.\end{proof}

\vspace{3mm}

\begin{thm}\label{them2} Let $T
:X\longrightarrow P^* (X )$  be a set-valued mapping and $U$ be a
random variable in the  stochastic approximation space $S=
(X,P)$. For any constants $a$ and $b$, we have
\[\underline{V}(aU+b ) = a^2\underline{E}(U^2)-(2a-c)(\underline{E}(U))^2+2b(a-c)\underline{E}(U)+b^2c~~~where~~~
c=\sum_{k=1}^{n}\underline{P}(u_k).\]
\end{thm}
\begin{proof} We have
\begin{eqnarray*}\underline{V}(aU+b
)&=&\underline{E}((aU+b)-\underline{E}(U))^2\\
&=&\underline{E}(aU+b)^2-2(aU+b)\underline{E}(U)+(\underline{E}(U))^2\\&=&
\underline{E}(a^2U^2+2abU+b^2-2aU\underline{E}(U)-2b\underline{E}(U)+(\underline{E}(U))^2)\\&=&
a^2\underline{E}(U^2)+2ab\underline{E}(U)+b^2c-2a(\underline{E}(U))^2-2bc\underline{E}(U)+c(\underline{E}(U))^2\\
&=&a^2\underline{E}(U^2)-(2a-c)(\underline{E}(U))^2+2b(a-c)\underline{E}(U)+b^2c, ~~~where~~~
c=\sum_{k=1}^{n}\underline{P}(u_k)
\end{eqnarray*}
\end{proof}

\vspace{3mm}

\begin{thm} Let $T
:X\longrightarrow P^* (X )$  be a set-valued mapping and $U$ be a
random variable in the  stochastic approximation space $S=
(X,P)$. For any constants $a$ and $b$, we have
\[\overline{V}(aU+b ) = a^2\overline{E}(U^2)-(2a-d)(\overline{E}(U))^2+2b(a-d)\overline{E}(U)+b^2d~~~where~~~
d=\sum_{k=1}^{n}\overline{P}(u_k).\] \end{thm}
\begin{proof}The proof is
similar to Theorem \ref{them2}.
\end{proof}

\end{document}